\documentclass[letterpaper,11pt]{article}

\usepackage{setspace}
\usepackage[utf8]{inputenc}
\usepackage[margin=1in]{geometry}
\usepackage{array}
\usepackage{ stmaryrd,amsxtra,amstext,amscd,amsfonts,fancyhdr}
\usepackage{ amssymb }
\usepackage{amsmath}
\usepackage{amsthm}
\usepackage{enumerate}

\usepackage{tikz}
\usetikzlibrary{calc}

\newcommand\del[1]{}

\allowdisplaybreaks

\makeatletter
\@namedef{subjclassname@2010}{%
  \textup{2010} Mathematics Subject Classification}
\makeatother

\newtheorem{theorem}{Theorem}[section]
\newtheorem{corollary}[theorem]{Corollary}
\newtheorem{proposition}[theorem]{Proposition}

\theoremstyle{definition}
\newtheorem{definition}[theorem]{Definition}
\newtheorem{lemma}[theorem]{Lemma}
\newtheorem{question}[theorem]{Question}
\newtheorem{example}[theorem]{Example}
\newtheorem{remark}[theorem]{Remark}

\numberwithin{equation}{section}

\newcommand\N{{\mathbb{N}}}
\newcommand\R{{\mathbb{R}}}

\newcommand\pd{{\partial}}

\renewcommand\L{{\mathcal{L}}}
\newcommand{\hk}{\mathbin{\! \hbox{\vrule height0.3pt width5pt depth 0.2pt \vrule height5pt width0.4pt depth 0.2pt}}}

\newcommand\Rho{\mathcal{P}}
\newcommand\g{\mathfrak{g}}
\newcommand\X{\mathfrak{X}}

\begin{document}


\title{N$\ddot{\text{o}}$ether's Theorem}
\author{Jon Herman}
\begin{center}
\setlength{\parindent}{0pt}
\setlength{\parskip}{17pt}

\vspace*{10pt}
\del{
{\Large \bf{Existence and Uniqueness of Weak moment Maps in Multisymplectic Geometry}\rm\\ 
\vspace{0.2cm}
\del{
\Large{by}
}

\vspace{0.2cm}

\Large{Jonathan Herman}}
}

{\Large \bf{Existence and Uniqueness of Weak Homotopy Moment Maps}\rm\\ 
\vspace{0.6cm}
\del{
\Large{by}
\vspace{-0.2cm}
}
\Large{Jonathan Herman}}
\\

{\it Department of Pure Mathematics, University of Waterloo}
\\

\tt{j3herman@uwaterloo.ca}

\vspace*{10pt}

\del{ 
{\small A research paper\\  
\vspace{0.1cm}
presented to the University of Waterloo\\
\vspace{0.1cm}in fulfilment of the\\
\vspace{0.1cm}research paper requirement for the degree of\\
\vspace{0.1cm}Master of Mathematics\\
\vspace{0.1cm}in\\
\vspace{0.1cm}Pure Mathematics

\vspace*{\fill}

Waterloo, Ontario, Canada, 2014

\copyright \  Jonathan Herman, 2014
}

}
\end{center}

 \del{

\cleardoublepage

 \pagenumbering{roman}
 \setcounter{page}{2}
 
\vspace*{0.5in}
\noindent
{   \setlength{\parindent}{0pt}
   \setlength{\parskip}{24pt}
   \setlength{\textwidth}{7in}

   {\sffamily\bfseries \index{copyright!author's declaration}
   AUTHOR'S DECLARATION}

   I hereby declare that I am the sole author of this research paper.  This is a true
   copy of the research paper, including any required final revisions, as accepted by
   my examiners.

   I understand that my research paper may be made electronically available to the
   public. 
   
   \vspace*{1.0in}
   
Jonathan Herman
}
}


\begin{abstract}

We show that the classical results on the existence and uniqueness of moment maps in symplectic geometry generalize directly to weak homotopy moment maps in multisymplectic geometry. In particular, we show that their existence and uniqueness is governed by a Lie algebra cohomology complex which reduces to the Chevalley-Eilenberg complex in the symplectic setup.  

\del{
We refine the definition of a homotopy moment map to generalize two concepts from Hamiltonian mechanics. In our setup, a Hamiltonian $(n-1)$-form is now preserved by its Hamiltonian vector field. Moreover, with our definition, we get that every element in the image of a moment map is a conserved quantity, as in Hamiltonian mechanics. We give results on the existence and uniqueness of these refined moment maps. In particular, we show that to each moment map there is a well defined cohomology class in the Lie algebra cohomology, analogous to the setup in symplectic geometry. 
}

\end{abstract}

\del{
\onehalfspacing
\cleardoublepage
\section*{Acknowledgements}
First and foremost, a very sincere thank you goes to my supervisor Dr. Spiro Karigiannis. I am extremely grateful for his patience and the tremendous amount of time he spent teaching and helping me this summer. I would also like to thank Dr. Shengda Hu, my second reader, for the very useful comments and corrections that he provided. I need to acknowledge two of my good friends; Janis Lazovksis and Cameron Williams. Janis is a LaTeX machine, and I am very appreciative of the time he spent helping me this summer. Cam was a great support this year, always there to help me work through any problem. I also would like to thank my family for their constant love and interest, it means very much to me. Last, but not least (I'd say least goes to Cameron Williams), a thank you goes to my cousin Matt Rappoport, for without our discussions some of the contents in this paper would not exist.
}


\del{

\cleardoublepage
\vspace*{70pt}
\begin{center}
\itshape OPTIONAL DEDICATION CAN GO HERE.
\end{center}

}

\tableofcontents

\pagenumbering{arabic}
\setcounter{page}{2}

\section{Introduction}

Recall that for a symplectic manifold $(M,\omega)$, a Lie algebra $\g$ is said to act symplectically if $\L_{V_\xi}\omega=0$, for all $\xi\in\g$, where $V_\xi$ is its infinitesimal generator. A symplectic group action is called Hamiltonian if one can find a moment map, that is, a map $f:\g\to C^\infty(M)$ satisfying \[df(\xi)=V_\xi\hk\omega,\] for all $\xi\in\g$.
\\

In multisymplectic geometry, $\omega$ is replaced by a closed, non-degenerate $(n+1)$-form, where $n\geq 1$. A Lie algebra action is called multisymplectic if $\L_{V_\xi}\omega=0$ for each $\xi\in\g$. A generalization of moment maps from symplectic to multisymplectic geometry is given by a (homotopy) moment map. These maps are discussed in detail in \cite{questions}. A homotopy moment map is a collection of maps, $f_k:\Lambda^k\g\to \Omega^{n-k}(M)$, with $1\leq k \leq n+1$, satisfying \begin{equation}\label{equation1}df_k(p)=-f_{k-1}(\partial_k(p))+(-1)^{\frac{k(k+1)}{2}}V_p\hk\omega,\end{equation}for all $p\in\Lambda^k\g$, where $V_p$ is its infinitesimal generator (see Definition \ref{inf gener}). A weak (homotopy) moment map is a collection of maps $f_k:\Rho_{\g,k}\to\Omega^{n-k}(M)$ satisfying  \begin{equation}\label{equation 2}df_k(p)=(-1)^{\frac{k(k+1)}{2}}V_p\hk\omega,\end{equation} for $p\in\Rho_{\g,k}$. Here $\Rho_{\g,k}$ is the Lie kernel, which is the kernel of the $k$-th Lie algebra cohomology differential $\partial_k:\Lambda^k\g\to\Lambda^{k-1}\g$,  defined by
\[\partial_k:\Lambda^k\g\to\Lambda^{k-1}\g \ \ \ \ \ \ \ \ \xi_1\wedge\cdots\wedge\xi_k\mapsto\sum_{1\leq i<j\leq k}(-1)^{i+j}[\xi_i,\xi_k]\xi_1\wedge\cdots\wedge\widehat\xi_i\wedge\cdots\wedge\widehat\xi_j\wedge\cdots\wedge\xi_k,\] for $k\geq 1$ and $\xi_1,\cdots,\xi_k\in \g$.

We see that any collection of functions satisfying equation (\ref{equation1}) must also satisfy (\ref{equation 2}). That is, any homotopy moment map induces a weak homotopy moment map.

 Weak moment maps generalize the moment maps of Madsen and Swann in \cite{ms} and \cite{MS}, and were also used to give a  multisymplectic version of Noether's theorem in \cite{me}. In this paper, we study the existence and uniqueness of weak homotopy moment maps and show that the theory is a generalization from symplectic geometry. We also show that the equivariance of a weak moment map can be characterized in terms of $\g$-module morphisms, analogous to symplectic geometry.
 
Recall that in symplectic geometry we have the following well-known results on the existence and uniqueness of moment maps. 

 \begin{proposition} Consider the symplectic action of a connected Lie group $G$ acting on a symplectic manifold $(M,\omega)$.
 
 \begin{itemize}
 \item  If the first Lie algebra cohomology vanishes, i.e. $H^1(\g)=0$, then a not necessarily equivariant moment map exists.
 
 \item If the second Lie algebra cohomology vanishes, i.e. $H^2(\g)=0$, then any non-equivariant moment map can be made equivariant.

\item  If the first Lie algebra cohomology vanishes, i.e. $H^1(\g)=0$, then equivariant moment maps are unique,
\end{itemize}

 and combining these results,
 
\begin{itemize}

\item  If both the first and second Lie algebra cohomology vanish, i.e. $H^1(\g)=0$ and $H^2(\g)=0$, then there exists a unique equivariant moment map.

\end{itemize}

\end{proposition}

 We generalize these results with the following theorems. Letting $\Omega^{n-k}_\mathrm{cl}$ denote the set of closed $(n-k)$-forms on $M$, we get the above propositions, in their respective order, by taking $n=k=1$.

\begin{theorem}
If $H^0(\g,\Rho_{\g,k}^\ast)=0$, then there exists a not necessarily equivariant weak homotopy $k$-moment map. The same result holds if $H^0(\g,\Rho_{\g,k}^\ast\otimes\Omega^{n-k}_{\mathrm{cl}})=0$ and $H^0(\g,\Omega^{n-k}_\mathrm{cl})\not=0$.
\end{theorem}
 
\begin{theorem}
If $H^1(\g,\Rho_{\g,k}^\ast\otimes\Omega^{n-k}_{\text{cl}})=0$, then any non-equivariant weak homotopy $k$-moment map can be made equivariant.
\end{theorem}

\begin{theorem}
If $H^0(\g,\Rho_{\g,k}^\ast\otimes\Omega^{n-k}_{\text{cl}})=0$ then an equivariant weak homotopy $k$-moment map is unique,
\end{theorem}
and combining these results,

\begin{theorem}
If $H^0(\g,\Rho_{\g,k}^\ast)=0$, and $H^1(\g,\Rho_{\g,k}^\ast\otimes\Omega^{n-k}_{\text{cl}})=0$, then there exists a unique equivariant weak $k$-moment map $f_k:\Rho_{\g,k}\to\Omega^{n-k}$. Moreover, if $H^0(\g,\Rho_{\g,k}^\ast)=0$,  and $H^1(\g,\Rho_{\g,k}^\ast\otimes\Omega^{n-k}_{\text{cl}})=0$ for all $1\leq k\leq n$, then a full equivariant weak moment map exists and is unique.
\end{theorem}

We also show that the morphism properties of moment maps from symplectic geometry are preserved in multisymplectic geometry. More specifically, recall that in symplectic geometry the equivariance of a moment map $f:\g\to C^\infty(M)$ is characterized by whether or not $f$ is a Lie algebra morphism. That is, $f$ is equivariant if and only if \[f([\xi,\eta])=\{f(\xi),f(\eta)\},\] for all $\xi,\eta\in\g$. However, as shown in Theorem 4.2.8 of \cite{Marsden} it is always true that $f$ induces a Lie algebra morphism between $\g$ and $C^\infty(M)/\text{constant}$, because $df([\xi,\eta])=d\{f(\xi),f(\eta)\}$. 

We generalize these results to multisymplectic geometry by showing that:

\begin{theorem}
For any $1\leq k \leq n$, a weak  $k$-moment map is always a $\g$-module morphism from $\Rho_{\g,k}\to\Omega^{n-k}_{\mathrm{Ham}}(M)/\text{closed}$. A weak $k$-moment map is equivariant if and only if it is a $\g$-module morphism from $\Rho_{\g,k}\to\Omega^{n-k}_{\mathrm{Ham}}(M)$.
\end{theorem}

Here $\Omega^{n-k}_{\mathrm{Ham}}(M)$ is denoting the space of multi-Hamiltonian forms, which are differential forms $\alpha\in\Omega^{n-k}(M)$ satisfying $d\alpha=X_\alpha\hk\omega$ for some $X_\alpha\in\Gamma(\Lambda^k(TM))$ (see Definition \ref{ham k form}). These forms were introduced in \cite{me}, and give a notion of a multi-symmetry which occurs when there is a given Hamiltonian $(n-1)$-form $H\in\Omega^{n-1}_{\mathrm{Ham}}(M)$, (see Definition \ref{ham 1 form}).
\section{Cohomology}
We briefly recall some basic notions from group and Lie algebra cohomology.
\subsection{Group Cohomology}
Let $G$ be a group and $S$ a $G$-module. For $g\in G$ and $s\in S$, let $g\cdot s$ denote the action of $G$ on $S$. Let $C^k(G,S)$ denote the space of smooth alternating functions from $G^k$ to $S$ and consider the differential $\partial_k:C^k(G,S)\to C^{k+1}(G,S)$ defined as follows. For $\sigma\in C^k(G,S)$ and $g_1,\cdots, g_{k+1}\in G$ define \begin{equation}\label{group differential}\begin{aligned}\partial&_k\sigma(g_1,\cdots,g_{k+1}):=\\
&g_1\cdot\sigma(g_2,\cdots,g_{k+1})+\sum_{i=1}^k(-1)^i\sigma(g_1,\cdots, g_{i-1},g_ig_{i+1},g_{i+2},\cdots,g_{k+1})-(-1)^k\sigma(g_1,\cdots,g_k).\end{aligned}\end{equation}A computation shows that $\partial_k^2=0$ so that $C^0(G,S)\to C^1(G,S)\to \cdots$ is a cochain complex.  This cohomology is known as the differentiable cohomology of $G$ with coefficients in $S$. We let $H^k(G,S)$ denote the $k$-th cohomology group and will call an equivalence class representative a $k$-cocycle.

\subsection{Lie Algebra Cohomology}
Let $\g$ be a Lie algebra and $R$ a $\g$-module. Given $\xi\in\g$ and $r\in R$, let $\xi\cdot r$ denote the action of $\g$ on $R$. We let $C^k(\g,R)$ denote the space of multilinear alternating functions from $\g^k$ to $R$ and consider the differential $\delta_k:C^k(\g,R)\to C^{k+1}(\g,R)$ defined as follows. For $f\in C^k(\g,R)$ and $\xi_1,\cdots\xi_{k+1}\in \g$ define

 \begin{equation}\label{group differential 2}\begin{aligned}\delta_k& f(\xi_1,\cdots,\xi_{k+1}):=\\
 &\sum_i(-1)^{i+1}\xi_i\cdot f(\xi_1,\cdots,\widehat\xi_i,\cdots,\xi_{k+1})+\sum_{i<j}(-1)^{i+j}f([\xi_i,\xi_j],\xi_1,\cdots,\widehat\xi_i,\cdots,\widehat\xi_j,\cdots,\xi_{k+1}). \end{aligned}\end{equation}

A computation shows that $\delta_k^2=0$. We let $H^k(\g,R)$ denote the $k$-th cohomology group and call an equivalence class representative a (Lie algebra) $k$-cocycle. Note that for $k=0$ the map $\delta_0:R\to  C^1(\g,R)$ is given by $(\delta_0r)(\xi)=\xi\cdot r$, where $r\in R$ and $\xi\in \g$. For $k=1$ the map $\delta_1:C^1(\g,R)\to C^2(\g,R)$ is given by $\delta_1(f)(\xi_1,\xi_2)=\xi_1\cdot f(\xi_2)-\xi_2\cdot f(\xi_1)-f([\xi_1,\xi_2])$, where $f\in C^1(\g,R)$ and $\xi_1$ and $\xi_2$ are in $\g$.\\

The standard example of Lie algebra cohomology is given when $R=\R$:

\begin{example}\bf{(Exterior algebra of $\g^\ast$) }\rm  Consider the trivial $\g$-action on $\R$. Then $C^k(\g,\R)=\Lambda^k\g^\ast$, and the Lie algebra cohomology differential $\delta_k:\Lambda^k\g^\ast\to\Lambda^{k+1}\g^\ast$ is given by
\begin{equation}\label{differential g dual}\delta_k\alpha(\xi_1\wedge\cdots\wedge\xi_k):= \alpha\left(\sum_{1\leq i<j\leq k}(-1)^{i+j}[\xi_i,\xi_k]\wedge\xi_1\wedge\cdots\wedge\widehat\xi_i\wedge\cdots\wedge\widehat\xi_j\wedge\cdots\wedge\xi_k)\right)\end{equation}where $\alpha\in\Lambda^k\g^\ast$, and $\xi_1\wedge\cdots\wedge\xi_k$ is a decomposable element of $\Lambda^k\g$, and extended by linearity to non-decomposables. It is easy to check that $\delta^2=0$. We will also make frequent reference to the corresponding Lie algebra homology differential which is given by

\begin{equation}\label{differential g}\partial_k:\Lambda^k\g\to\Lambda^{k-1}\g \ \ \ \ \ \ \ \ \xi_1\wedge\cdots\wedge\xi_k\mapsto\sum_{1\leq i<j\leq k}(-1)^{i+j}[\xi_i,\xi_k]\wedge\xi_1\wedge\cdots\wedge\widehat\xi_i\wedge\cdots\wedge\widehat\xi_j\wedge\cdots\wedge\xi_k,\end{equation}for $k\geq 1$. We define $\Lambda^{-1}\g=\{0\}$ and $\partial_0$ to be the zero map. 
\end{example} 

For the rest of this section we only consider the exterior algebra homology complex.

\begin{definition}\label{cohomology differential}
We follow the terminology and notation of \cite{ms} and call $\Rho_{\g,k}=\ker \partial_k$ the $k$-th Lie kernel, which is a vector subspace of $\Lambda^k\g$.  Notice that if $\g$ is abelian then $\Rho_{\g,k}=\Lambda^k\g$. We will let $\Rho_g$ denote the direct sum of all the Lie kernels; \[\Rho_\g=\oplus_{k=0}^{\dim(\g)}\Rho_{\g,k},\] and denote $H^k(\g,\R)$  simply by $H^k(\g)$.
\end{definition}

We now recall the Schouten Bracket.

\begin{definition}
On decomposable multivectors $X=X_1\wedge\cdots \wedge X_k \in\Lambda^k \g$ and $Y=Y_1\wedge\cdots \wedge Y_l\in\Lambda^l\g$,  the Schouten bracket $[\cdot,\cdot]$ is given by \[ [X,Y]:=\sum_{i=1}^k\sum_{j=1}^l(-1)^{i+j}[X_i,Y_j]\wedge X_1\wedge\cdots\wedge \widehat X_i\wedge\cdots\wedge X_k\wedge Y_1\wedge\cdots\wedge\widehat Y_j\wedge\cdots\wedge Y_l,\] and extended by linearity to all multivector fields. \del{The Schouten bracket is also defined on the multivector fields $\Gamma(\Lambda^\bullet(TM))$ by the same formula. We will abuse notation and let $[\cdot,\cdot]$ represent the Schouten bracket on both $\Lambda^\bullet\g$ and $\Gamma(\Lambda^\bullet(TM))$. }
\end{definition}
The next proposition shows that the Schouten bracket and the Lie algebra differential are equal, when restricted to elements of a certain form.

\begin{proposition}
\label{wedge is Schouten}
For $p\in\Rho_{\g,k}$ and $\xi\in\g$ we have that \[\partial(p\wedge\xi)=[p,\xi].\]
\end{proposition}

\begin{proof}
A computation using the definition of $\partial$ shows that
\begin{align*}
\partial(p\wedge \xi)&=\partial(p)\wedge\xi +p\wedge\partial(\xi)+[p,\xi]\\
&=[p,\xi]&\text{since $p\in\Rho_{\g,k}$ and $\partial(\xi)=[\xi,\xi]=0$. }
\end{align*}
\end{proof}

Let $\g$ be a Lie algebra acting on a manifold $M$. For $\xi\in\g$, we let $V_\xi\in\Gamma(TM)$ denote its infinitesimal generator. 

\begin{definition}\label{inf gener} For a decomposable element $p=\xi_1\wedge\cdots\wedge\xi_k$ of $\Lambda^k\g$, its infinitesimal generator, denoted $V_p$, is the multivector field $V_{\xi_1}\wedge\cdots\wedge V_{\xi_k}$.
\end{definition}

\begin{lemma}(\bf{Extended Cartan Lemma})\rm
\label{extended Cartan}
For a decomposable multivector field $p=\xi_1\wedge\cdots\wedge\xi_k$ in $\Lambda^k\g$ and differential form $\tau$ we have that 

\begin{align*}
(-1)^kd(V_p\hk\tau)&=V_{\partial{p}}\hk\tau +\sum_{i=1}^k(-1)^i(V_{\xi_1}\wedge\cdots\wedge\widehat V_{\xi_i}\wedge\cdots\wedge V_{\xi_k})\hk\L_{V_{\xi_i}}\tau +V_p\hk d\tau.
\end{align*}
\end{lemma}

\begin{proof}

This is Lemma 3.4 of \cite{ms} or Lemma 2.18 of \cite{cq}.
\end{proof}

Let $\Phi:G\times M\to M$ be a Lie group action on $M$. 
\begin{definition}
For $A\in\Gamma(TM)$ we let $\Phi_{g}^\ast A$ denote the vector field given by the push-forward of $A$ by $\Phi_{g}^{-1}$. That is, \[(\Phi_{g}^\ast(A))_x:=(\Phi_{g^{-1}})_{\ast,\Phi_{g}(x)}(A_{\Phi_{g(x)}}),\] where $x\in M$. For a decomposable multivector field $Y=Y_1\wedge\cdots \wedge Y_k$ in $\Gamma(\Lambda^k(TM))$ we will let $\mathrm{Ad}_gY$ denote the extended adjoint action \[\mathrm{Ad}_gY=\mathrm{Ad}_gY_1\wedge\cdots\wedge \mathrm{Ad}_gY_k\] and we will let $\Phi_g^\ast Y$ denote the multivector field \[\Phi_g^\ast Y=\Phi_g^\ast Y_1\wedge\cdots\wedge \Phi_g^\ast Y_k.\]We also extend $\mathrm{ad}$ to a map $\mathrm{ad}:\g\times \Lambda^k\g\to\Lambda^k\g$ by \begin{equation}\label{ad equation}\mathrm{ad}_\xi(Y_1\wedge\cdots\wedge Y_k)=\sum_{i=1}^kY_1\wedge\cdots\wedge \mathrm{ad}_{\xi}(Y_i)\wedge\cdots\wedge Y_k.\end{equation}
\end{definition}

\begin{corollary}
\label{ad}
For $\xi\in\g$ we have that $\mathrm{ad}_\xi$ preserves the Lie kernel. That is, if $p$ is in $\Rho_{\g,k}$ then $\mathrm{ad}_{\xi}(p)$ is in $\Rho_{\g,k}$.
\end{corollary}

\begin{proof}

A computation shows that $\mathrm{ad}_\xi(p)=[\xi,p]$. Hence, by Proposition \ref{wedge is Schouten}, we see $[\xi,p]$ is exact. Thus it is closed.
\end{proof}

The next proposition shows that the infinitesimal generator of the extended adjoint action agrees with the pull back action.
\begin{proposition}
\label{adjoint over wedge}
Let $\Phi:G\times M\to M$ be a group action. For every $g\in G$ and $p\in\Lambda^k\g$ we have that \[V_{Ad_gp}=\Phi_{g^{-1}}^\ast V_p.\] Equivalently, the map $\Lambda^k\g\to\Gamma(\Lambda^k(TM))$ given by $\xi_1\wedge\cdots\wedge\xi_k\mapsto V_{\xi_1}\wedge\cdots\wedge V_{\xi_k}$ is equivariant with respect to the extended adjoint and pull back action.
\end{proposition}

\begin{proof}
Fix $q\in M$, $g\in G$. First suppose that $\xi\in\g$. Then by Proposition 4.1.26 of \cite{Marsden} we have that \[V_{Ad_g\xi}=\Phi_{g^{-1}}^\ast V_\xi.\] The claim now follows since for $p=\xi_1\wedge\cdots\wedge \xi_k$ in $\Lambda^k\g$,
\begin{align*}
V_{Ad_gp}&:=V_{Ad_g{\xi_1}}\wedge\cdots\wedge V_{Ad_g{\xi_k}}\\
&=\Phi_{g^{-1}}^\ast V_{\xi_1}\wedge\cdots\wedge\Phi_{g^{-1}}^\ast V_{\xi_k}\\
&=\Phi_{g^{-1}}^\ast V_p&\text{by definition.}
\end{align*}

\end{proof}

\section{Multisymplectic Geometry}

Here we recall some concepts and tools used in multisymplectic geometry.

\subsection{Multisymplectic Manifolds}
\begin{definition}
A manifold $M$ equipped with a closed $(n+1)$-form $\omega$ is called a pre-multisymplecitc manifold. If in addition the map $T_pM\to\Lambda^n T_p^\ast M,\  V\mapsto V\hk\omega$ is injective, then $(M,\omega)$ is called a multisymplectic, or $n$-plectic, manifold.
\end{definition} 

\begin{definition}\label{ham 1 form}
An $(n-1)$-form $\alpha\in\Omega^{n-1}(M)$ is called Hamiltonian if there exists a vector field $V_\alpha\in\Gamma(TM)$ such that $d\alpha=-V_\alpha\hk\omega$. Note that the non-degeneracy of $\omega$ insures uniqueness of the corresponding Hamiltonian vector field. We let $\Omega^{n-1}_{\mathrm{Ham}}(M)$ denote the space of Hamiltonian $(n-1)$-forms, which is a subspace of $\Omega^{n-1}(M)$.
\end{definition}
As in symplectic geometry, we are interested in Lie group actions which preserve the $n$-plectic form.

\begin{definition}
A Lie group action $\Phi:G\times M\to M$ is called multisymplectic if $\Phi_g^\ast\omega=\omega$. A Lie algebra action $\g\times \Gamma(TM)\to \Gamma(TM)$ is called multisymplectic if $\L_{V_\xi}\omega=0$ for all $\xi\in\g$. We remark that a multisymplectic Lie group action induces a multisymplectic Lie algebra action. Conversely, a multisymplectic Lie algebra action induces a multisymplectic group action if the Lie group is connected.  
\end{definition}

In \cite{rogers} it was shown that to any multisymplectic manifold one can associate the following $L_\infty$-algebra.
\begin{definition}\label{Lie n observables}
The Lie $n$-algebra of observables, $L_\infty(M,\omega)$ is the following $L_\infty$-algebra. Let $L=\oplus_{i=0}^nL_i$ where $L_0=\Omega^{n-1}_{\text{Ham}}(M)$ and $L_i=\Omega^{n-1-i}(M)$ for $1\leq i\leq n-1$. The maps $l_k:L^{\otimes k}\to L$ of degree $k-2$ are defined as follows: For $k=1$,

\[l_1(\alpha) =\left\{\begin{array}{l}d\alpha \ \ \ \ \ \text{if deg $\alpha>0$},\\
0 \ \ \ \ \ \ \ \text{if deg $\alpha=0$.}\end{array}\right.\]

For $k>1$,
\[l_k(\alpha_1,\cdots,\alpha_k) =\left\{\begin{array}{l}\zeta(k)X_{\alpha_k}\hk\cdots X_{\alpha_1}\hk\omega \ \ \ \ \ \text{if deg $\alpha_1\otimes\cdots\otimes\alpha_k=0$},\\
0 \ \ \ \ \ \ \ \ \ \ \ \ \ \ \ \ \ \ \ \ \ \ \ \ \ \ \ \ \ \ \text{if deg $\alpha_1\otimes\cdots\otimes\alpha_k>0$.}\end{array}\right.\]
Here $\zeta(k)$ is defined to equal $-(-1)^{\frac{k(k+1)}{2}}$. We introduce this notation as this sign comes up frequently.

\end{definition}

\subsection{Hamiltonian forms}
Let $(M,\omega)$ be an $n$-plectic manifold. The following definition generalizes the concept of a Hamiltonian $1$-form from symplectic geometry.

\begin{definition}\label{ham k form}
A differental form $\alpha\in\Omega^{n-k}(M)$ is called Hamiltonian if there exists a multivector field $X_\alpha\in\Gamma(\Lambda^k(TM))$ such that $d\alpha=-X_\alpha\hk\omega$. 
\end{definition}

Note that the Hamiltonian multivector field corresponding to a Hamiltonian form is not unique; however, the difference of any two Hamiltonian vector fields is in the kernel of $\omega$. The next proposition shows that the Hamiltonian forms are an $L_\infty$ subalgebra of the Lie $n$-algebra of observables 

\begin{proposition}
Let $\widehat L_i=\Omega^{n-i}_{\mathrm{Ham}}(M)$ and $\widehat L=\oplus_{i=0}^n\widehat L_i$. Let $\widehat L_\infty(M,\omega)$ denote the space $\widehat L$ together with the mappings $l_k$ defined above in the definition of the Lie-$n$-algebra of observables. Then $\widehat L_\infty(M,\omega)$ is an $L_\infty$-subalgebra of $L_\infty(M,\omega)$.
\end{proposition}

\begin{proof}

This is Theorem 4.15 of \cite{me}.

\end{proof}
\subsection{Weak Homotopy Moment Maps}

For a group acting on a symplectic manifold $M$, a moment map is a Lie algebra morphism between $(\g,[\cdot,\cdot])$ and  $(C^\infty(M),\{\cdot,\cdot\})$, where $\{\cdot,\cdot\}$ is the Poisson bracket. In multisymplectic geometry, a moment map is an $L_\infty$-morphism from the exterior algebra of $\g$ to the Lie $n$-algebra of observables. We direct the reader to \cite{questions} for more information on $L_\infty$-algebras and morphisms.
\begin{definition}\label{hmm}
A (homotopy) moment map is an $L_\infty$-morphism $(f)$ between $\g$ and the Lie $n$-algebra of observables. This means that $(f)$ is a collection of maps $f_1:\Lambda^1\g\to\Omega^{n-1}_{\text{Ham}}(M)$ and $f_k:\Lambda^k\g\to\Omega^{n-k}(M)$ for $k\geq 2$ satisfying, for $p\in\Lambda^k\g$ \begin{equation}\label{hcmm}-f_{k-1}(\partial p)=df_k(p)+\zeta(k)V_p\hk\omega.\end{equation}
\end{definition}

It follows immediately from equation (\ref{hcmm}) that if $p$ is in $\Rho_{\g,k}$ then $f_k(p)$ is a Hamiltonian form. That is, if the domain of a homotopy moment map $(f)$ is restricted to the Lie kernel, then the image of $(f)$ is completely contained in the space of Hamiltonian forms. This motivates the definition of a weak homotopy moment map:

\begin{definition} A weak (homotopy) moment map, is a collection of maps $(f)$ with $f_k:\Rho_{\g,k}\to\Omega_{\mathrm{Ham}}^{n-k}(M)$ satisying \begin{equation}\label{hcmm kernel}df_k(p)=-\zeta(k)V_p\hk\omega.\end{equation}We refer to the component $f_k$ as a weak $k$-moment map.
\end{definition}
\del{
Notice that equation (\ref{hcmm kernel}) is directly analogous to the defining equation of a moment map in symplectic geometry. Moreover, it was shown in \cite{cq} that elements in the Lie kernel get mapped to conserved quantities, just like in symplectic geometry. Moreover, in \cite{me} it was shown that if one consider weak moment maps, then one can obtain a generalization of Noether's theorem to multisymplectic geometry.
}

\begin{remark}
Notice that any moment map gives a weak moment map. Indeed, if $(f)$ satisfies equation (\ref{hcmm}) then it satisfies equation (\ref{hcmm kernel}).
\end{remark}
\begin{remark}
Notice that a weak homotopy moment map coincides with the moment map from symplectic geometry in the case $n=1$. Indeed, setting $n=1$ in equations (\ref{hcmm}) and (\ref{hcmm kernel}) yields $f:\g\to C^\infty(M)$ such that $df(\xi)=-V_\xi\hk\omega$. Also notice that the $n$-th component of a weak moment map is precisely the moment map introduced by Madsen and Swann in \cite{ms} and \cite{MS}.
\end{remark}

The next proposition says that a weak moment map is still an $L_\infty$-morphism.
\begin{proposition}
A weak moment map is an $L_\infty$-morphism from $\g$ to $\widehat L_\infty(M,\omega)$.
\end{proposition}

\begin{proof}
This is Proposition 5.9 of \cite{me}.
\end{proof}

\begin{definition}
A homotopy moment map $(f)$ is equivariant if each component $f_k:\Lambda^k\g\to \Omega^{n-k}(M)$ is equivariant with respect to the adjoint and pullback actions respectively. That is, for all $g\in G$, $p\in\Lambda^k\g$ and $1\leq k\leq n$ \begin{equation}\label{equivariant equation}f_k(Ad_{g^{-1}}^\ast p)=\Phi_g^\ast f(p).\end{equation} Similarly, a weak moment map is equivariant if equation (\ref{equivariant equation}) holds for all $p\in\Rho_{\g,k}$ and $1\leq k\leq n$.
\end{definition}

We study the equivariance of moment maps further in the following section.

\section{Equivariance of Weak Moment Maps}

In this section we show how the theory of equivariance of moments maps in symplectic geometry generalizes to multisymplectic geometry.

\subsection{Equivariance in Multisymplectic Geometry}

We first recall the theory from symplectic geometry without proof and then generalize to the multisymplectic setting. The results from symplectic geometry can all be found in Chapter 4.2 of \cite{Marsden} for example.  Let $(M,\omega)$ be a symplectic manifold, and $\Phi:G\times M\to M$ a symplectic Lie group action by a connected Lie group $G$ . We consider the induced symplectic Lie algebra action $\g\times\Gamma(TM)\to\Gamma(TM)$. Suppose that a moment map  $f:\g\to C^\infty(M)$ exists. That is, $df(\xi)=V_\xi\hk\omega$ for all $\xi\in\g$. By definition, $f$ is equivariant if \[f(\mathrm{Ad}_{g^{-1}}\xi)=\Phi_g^\ast f(\xi).\]Following Chapter 4.2 of \cite{Marsden}, for $g\in G$ and $\xi\in\g$ define $\psi_{g,\xi}\in C^\infty(M)$ by \begin{equation}\label{psi}\psi_{g,\xi}(x):= f(\xi)(\Phi_g(x))-f(\mathrm{Ad}_{g^{-1}}\xi)(x).\end{equation}

\begin{proposition}\label{psi constant}
For each $g\in G$ and $\xi\in\g$, the function $\psi_{g,\xi}\in C^\infty(M)$ is constant.
\end{proposition}
Since $\psi_{g,\xi}$ is constant, we may define the map $\sigma:G\to\g^\ast$ by \[\sigma(g)(\xi):=\psi_{g,\xi},\] where the right hand side is the constant value of $\psi_{g,\xi}$.
\del{This motivates consideration of the chain complex \[\g^\ast\otimes\R\to(G\to \g^\ast\otimes\R)\to(G\times G\to\g^\ast\otimes\R)\to\cdots,\]with the natural action of $G$ on $\g^\ast\otimes\R$ given by \[g\cdot\alpha(\xi):=Ad_{g^{-1}}^\ast\alpha(\xi)\] for $g\in G,\alpha\in\g^\ast\otimes\R$ and $\xi\in\g$.}
\begin{proposition}\label{cocycle}
The map $\sigma:G\to\g^\ast$ is a cocycle in the chain complex \[\g^\ast\to C^1(G,\g^\ast)\to C^2(G,\g^\ast)\to\cdots.\]That is, $\sigma(gh)=\sigma(g)+\mathrm{Ad}_{g^{-1}}^\ast\sigma(h)$ for all $g,h\in G$. 
\end{proposition}
The map $\sigma$ is called the cocycle corresponding to $f$. The following proposition shows that for any symplectic group action, the cocycle gives a well defined cohomology class. 

\begin{proposition}\label{cohomology class}
For any symplectic action of  $G$ on $M$ admitting a moment map, there is a well defined cohomology class. More specifically, if $f_1$ and $f_2$ are two moment maps, then their corresponding cocycles $\sigma_1$ and $\sigma_2$ are in the same cohomology class, i.e. $[\sigma_1]=[\sigma_2]$.
\end{proposition}

By definition, we see that $\sigma$ is measuring the equivariance of $f$. That is, $\sigma=0$ if and only if $f$ is equivariant. Moreover, if the cocycle corresponding to a moment map vanishes in cohomology, the next proposition shows that we can modify the original moment map to make it equivariant.

\begin{proposition}\label{sigma class zero}
Suppose that $f$ is a moment map with corresponding cocycle $\sigma$. If $[\sigma]=0$ then $\sigma=\partial\theta$ for some $\theta\in\g^\ast$, and $f+\theta$ is an equivariant moment map.
\end{proposition}

We now show how this theory generalizes to multisymplectic geometry. For the rest of this section we let $(M,\omega)$ denote an $n$-plectic manifold and $\Phi:G\times M\to M$ a multisymplectic connected group action. We consider the induced multisymplectic Lie algebra action $\g\times\Gamma(TM)\to\Gamma(TM)$. Assume that we have a weak homotopy moment map $(f)$, i.e. a collection of maps $f_k:\Rho_{\g,k}\to \Omega^{n-k}_{\mathrm{Ham}}(M)$ satisfying equation (\ref{hcmm kernel}). 

To extend equation (\ref{psi}) to multisymplectic geometry, for $g\in G$ and $p\in\Rho_{\g,k}$, we define the following $(n-k)$-form: \begin{equation}\label{ms equiv}\psi^k_{g,p}:= f_k(p)-\Phi_{g^{-1}}^\ast f_k(\mathrm{Ad}_{g^{-1}}(p)).\end{equation}The following proposition generalizes Proposition \ref{psi constant}. 

\begin{proposition}\label{general closed}
The $(n-k)$-form $\psi^k_{g,p}$ is closed.
\end{proposition}
\begin{proof}

\del{Since our moment maps are taking values in the special forms, it follows that $\psi_{g,p}$ is a special form as the sum of two special forms. It remains to show that $\psi_{g,p}$ is closed. } Since $\Phi^\ast_g$ is injective and commutes with the differential, our claim is equivalent to showing that $\Phi^\ast_g(\psi^k_{g,p})$ is closed. Indeed we have that
\begin{align*}
d(\Phi_g^\ast(\psi^k_{g,p}))&=d(\Phi_g^\ast f_k(p)-f_k((\mathrm{Ad}_{g^{-1}}p)))\\
&=\Phi_g^\ast(df_k(p))-d(f_k(\mathrm{Ad}_{g^{-1}}(p)))\\
&=-\zeta(k)\Phi_g^\ast(V_p\hk\omega)+\zeta(k)V_{\mathrm{Ad}_{g^{-1}}p}\hk\omega&\text{since $(f)$ is moment map}\\
&=-\zeta(k)\Phi_g^\ast(V_p\hk\omega)+\zeta(k)(\Phi_g^\ast V_p)\hk\omega&\text{by Proposition \ref{adjoint over wedge}}\\
&=-\zeta(k)\Phi_g^\ast(V_p\hk\omega)+\zeta(k)\Phi_g^\ast(V_p\hk\omega)&\text{since $G$ preserves $\omega$}\\
&=0.
\end{align*}
\end{proof}

\del{
\begin{definition}
To ease the notation, we will let $\Omega_{Sc}^k(M)$ denote the intersection of $\Omega^k_{S}(M)$ and $\Omega^k_{\text{cl}}(M)$. 
\end{definition}
}

In analogy to symplectic geometry, we now show that each component of a weak moment map gives a cocycle.
\begin{definition} We call the map $\sigma_k:G\to\Rho_{\g,k}^\ast\otimes\Omega^{n-k}_{\text{cl}}$ defined by \[\sigma_k(g)(p):=\psi^k_{g,p}\]the cocycle corresponding to $f_k$. 
\end{definition}

As a generalization of Proposition \ref{cocycle} we obtain:

\begin{proposition}
\label{multi cocycle}
The map $\sigma_k$ is a $1$-cocycle in the chain complex \[\Rho_{\g,k}^\ast\otimes\Omega^{n-k}_{\mathrm{cl}}\to C^1(G,\Rho_{\g,k}^\ast\otimes\Omega^{n-k}_{\mathrm{cl}}) \to C^2(G,\Rho_{\g,k}^\ast\otimes\Omega^{n-k}_{\mathrm{cl}}) \to \cdots,\]where the action of $G$ on $\Rho_{\g,k}^\ast\otimes\Omega^{n-k}_{\mathrm{cl}}$ is given by the tensor product of the co-adjoint and pullback actions. The induced infinitesimal action of $\g$ on $\Rho_{\g,k}^\ast\otimes\Omega^{n-k}_{\mathrm{cl}}$ is defined as follows: for $f\in\Rho_{\g,k}^\ast\otimes\Omega^{n-k}_{\mathrm{Ham}}$, $p\in\Rho_{\g,k}$ and $\xi\in \g$, \begin{equation}(\xi\cdot f)(p):=f(\mathrm{ad}_\xi(p))+\L_{V_\xi}f(p).\end{equation}
\end{proposition}

\begin{proof}
By equation (\ref{group differential}) we know that $(\partial(\sigma)(g,h))(p):=\sigma(gh)(p)-\sigma(g)(p)-g\cdot\sigma(h)(p)$. For arbitrary $p\in\Rho_{\g,k}$ we have
\begin{align*}
\sigma_k(gh)(p)&=f_k(p)-\Phi^\ast_{(gh)^{-1}}(f_k\mathrm{Ad}_{(gh)^{-1}}p)\\
&=f_k(p)-\Phi^\ast_{g^{-1}}\Phi^\ast_{h^{-1}}(f_k((\mathrm{Ad}_{h^{-1}}\mathrm{Ad}_{g^{-1}}p)))\\
&=f_k(p)-\Phi^\ast_{g^{-1}}(f_k(\mathrm{Ad}_{g^{-1}}p))+\Phi^\ast_{g^{-1}}(f_k(\mathrm{Ad}_{g^{-1}}p))-\Phi^\ast_{g^{-1}}(\Phi^\ast_{h^{-1}}(f_k(\mathrm{Ad}_{h^{-1}}\mathrm{Ad}_{g^{-1}}p)))\\
&=\sigma_k(g)(p)+\Phi_{g^{-1}}^\ast(\sigma_k(h)(\mathrm{Ad}_{g^{-1}}p))\\
&=\sigma_k(g)(p)+g\cdot \sigma_k(h)(p).
\end{align*}
\end{proof}

\begin{definition} Let \[\mathfrak{C}=\bigoplus_{k=1}^n \Rho_{\g,k}^\ast\otimes\Omega^{n-k}_{\text{cl}}\] Let $\sigma=\sigma_1+\sigma_2+\cdots $. We call the map $\sigma$ the cocycle corresponding to $(f)$. 
\end{definition}
Since the components of a weak moment map do not interact, as a corollary to Proposition \ref{multi cocycle} we obtain
\begin{proposition}
The map $\sigma$ is a cocycle in the complex \[\mathfrak{C}\to C^1(G,\mathfrak{C})\to C^2(G,\mathfrak{C})\to\cdots\]

\end{proposition}
The next theorem shows that multisymplectic  Lie algebra actions admitting weak moment maps give a well defined cohomology class, generalizing Proposition \ref{cohomology class}.

\begin{theorem}\label{general class}
Let $G$ act multisymplectically on $(M,\omega)$. To any weak moment map, there is a well defined cohomology class $[\sigma]$ in $H^1(G,\mathfrak{C})$. More precisely if $(f)$ and $(g)$ are two weak moment maps with cocycles $\sigma$ and $\tau$, then $\sigma-\tau$ is a coboundary.
\end{theorem}

\begin{proof}
We need to show that $\sigma_k-\tau_k$ is a coboundary for each $k$. We have that \[\sigma_k(g)(p)-\tau_k(g)(p)=f_k(p)-g_k(p)-\Phi_{g^{-1}}^\ast(f_k(Ad_{g^{-1}}p)-g_k(Ad_{g^{-1}}(\xi))).\]
However, $(f)$ and $(g)$ are both moment maps and so $d(f_k(p)-g_k(p))=0$. Thus $f_k-g_k$ is in $\mathfrak{C}$. Moreover, by equation (\ref{group differential 2}), we see that $\sigma_k-\tau_k=\partial(f_k-g_k)$.

\end{proof}

If $(f)$ is not equivariant but its cocycle vanishes, then we can define a new equivariant moment map from $(f)$, in anology to Proposition \ref{sigma class zero}. 

\begin{proposition}\label{general make}
Let $(f)$ be a weak moment map with cocycle satisfying $[\sigma]=0$. This means that $\sigma=\partial\theta$ for some $\theta\in\mathfrak{C}$. The map $(f)+\theta$ is a weak moment map that is equivariant.
\end{proposition}

\begin{proof}
We have that $(f)+\theta$ is a moment map since $\theta(p)$ is closed for all $p\in\Rho_{\g,k}$.  Let $\widetilde\sigma$ denote the corresponding cocycle. Note that by equation (\ref{group differential 2}) we have $(\partial(\theta)(g))(p)=\theta(\mathrm{Ad}_{g^{-1}}p)-\Phi_g^\ast\theta(p)$. By the injectivity of $\Phi_g^\ast$, to show that $\widetilde\sigma=0$, it is sufficient to show that $\Phi_g^\ast(\widetilde\sigma(g)(p))=0$ for all $g\in G$ and $p\in\Rho_{g,k}$. Indeed, \begin{align*}
\Phi_g^\ast(\widetilde\sigma_k(g)(p))&=\Phi_g^\ast f(p)+\Phi_g^\ast\theta(p)-f(\mathrm{Ad}_{g^{-1}}p)-\theta(\mathrm{Ad}_{g^{-1}}p)\\
&=\sigma(g)(\xi)-\partial\theta(g)(\xi)\\
&=\sigma(g)(\xi)-\sigma(g)(\xi)&\text{since $\partial\theta =\sigma$}\\
&=0.
\end{align*}
\end{proof}

If $(f)$ is not equivariant with respect to the $G$-action, then we can define a new action for which $(f)$ is equivariant.

\begin{proposition}
\label{new action}
For $g\in G$ define $\Upsilon_g:\Rho_{\g,k}^\ast\otimes\Omega^{n-k}_{\mathrm{Ham}}\to\Rho_{\g,k}^\ast\otimes\Omega^{n-k}_{\mathrm{Ham}}$ by 

\[\Upsilon_g(\theta)(p):= \Phi_{g^{-1}}^\ast \theta(\mathrm{Ad}_{g^{-1}}p)+\sigma(g)(p)\]where $\theta$ is in $\Rho_{\g,k}^\ast\otimes\Omega^{n-k}_{\mathrm{Ham}}$ and $p$ is in $\Rho_{\g,k}$. Then $\Upsilon_g$ is a group action and $(f)$ is $\Upsilon_g$-equivariant.

\end{proposition}

\begin{proof}
The proof is a direct extension from the proof of Proposition 4.2.7 in \cite{Marsden}. We first show that $\Upsilon_g$ is a group action. Indeed, $\sigma(e)=0$ and $\mathrm{Ad}_{e}$ is the identity showing that $\Upsilon_e(\theta)=\theta$. For the multiplicative property of the group action we have
\begin{align*}
\Upsilon_{gh}(\theta)(p)&=\Phi^\ast_{(gh)^{-1}}\theta(\mathrm{Ad}_{(gh)^{-1}}p)+\sigma(gh)(p)\\
&=\Phi^\ast_{g^{-1}}(\Phi^\ast_{h^{-1}}(\theta(\mathrm{Ad}_{h^{-1}}Ad_{g^{-1}}p)))+\sigma(g)(p)+\Phi^\ast_{g^{-1}}(\sigma(h)(\mathrm{Ad}_{g^{-1}}p))&\text{by Proposition \ref{cocycle}}\\
&=\Phi^\ast_{g^{-1}}(\Upsilon_h(\theta)(\mathrm{Ad}_{g^{-1}}p))+\sigma(g)(p)\\
&=\Upsilon_g(\Upsilon_h(\theta))(p).
\end{align*} To show that $f_k$ is equivariant
The moment map $f_k$ is equivariant with respect to this action because 
\begin{align*}
\Upsilon_g(f_k)(p)&=\Phi^\ast_{g^{-1}}(f_k(\mathrm{Ad}_{g^{-1}}p))+\sigma(g)(p)\\
&=\Phi^\ast_{g^{-1}}f_k(\mathrm{Ad}_{g^{-1}}p)+ f_k(p)-\Phi^\ast_{g^{-1}}f_k(\mathrm{Ad}_{g^{-1}}p)\\
&=f_k(p).
\end{align*}
\end{proof}

\subsection{Infinitesimal Equivariance in Multisymplectic Geometry}

Next we recall the notion of infinitesimal version of equivariance in symplectic geometry. That is, we differentiate equation (\ref{psi}) to obtain the map $\Sigma:\g\times\g\to C^\infty(M)$ defined  by $\Sigma(\xi,\eta):=\left.\frac{d}{dt}\right|_{t=0}\psi_{\exp(t\eta),\xi}$. A straightforward computation, which we generalize in Proposition \ref{comp of Sigma}, gives that \[\Sigma(\xi,\eta)=f([\xi,\eta])-\{f(\xi),f(\eta)\}.\] Another quick computation shows that $df([\xi,\eta])=d\{f(\xi),f(\eta)\}$, showing $\Sigma(\xi,\eta)$ is a constant function for every $\xi,\eta\in\g$. That is, $\Sigma$ is a function from $\g\times\g$ to $\R$. 

\begin{proposition}\label{inf cocycle}The map $\Sigma:\g\times\g\to\R$ is a Lie algebra $2$-cocycle in the chain complex \[\R\to C^1(\g, \R)\to C^2(\g,\R)\to\cdots.\] 
\end{proposition}

\begin{definition}\label{inf equiv moment}
A moment map $f:\g\to C^\infty(M)$ is infinitesimally equivariant if $\Sigma=0$, i.e. if \begin{equation}\label{equivariant equation 2}f([\xi,\eta])=\{f(\xi),f(\eta)\}\end{equation} for all $\xi,\eta\in\g$.
\end{definition}

\begin{proposition}
For a connected Lie group, infinitesimal equivariance and equivariance are equivalent.
\end{proposition}

\begin{proof}
This is clear since $\Sigma$ is just the derivative of $\sigma$.
\end{proof}

Since we will always be working with connected Lie groups, we will abuse terminology and call a moment map equivariant if it satisfies equation (\ref{equivariant equation}) or (\ref{equivariant equation 2}).
\del{
\begin{proposition}\label{lie alg morph 1}
A pre-moment map is a Lie algebra morphism from $(\g,[\cdot,\cdot])$ to $(C^\infty(M)/\mathrm{closed},\{\cdot,\cdot\})$.
\end{proposition} 
\begin{proof}
Since $df([\xi,\eta])=d\{f(\xi),f(\eta)\}$ it follows that $f([\xi,\eta])-\{f(\xi),f(\eta)\}$ is a constant. Hence, for a pre-moment map it is always true that $f([\xi,\eta])=\{f(\xi),f(\eta)\}$ in the quotient space.
\end{proof}

\begin{proposition}\label{lie alg morph 2}
If $f$ is an equivariant moment map then $f$ is a Lie algebra morphism from $(\g,[\cdot,\cdot])$ to $(C^\infty(M),\{\cdot,\cdot\})$.
\end{proposition}
\begin{proof}
This follows from the above, since if $f$ is equivariant, then $\Sigma=0$.
\end{proof}

\begin{remark}
The converse to Proposition \ref{lie alg morph 2} is not necessarily true. That is, if a moment map is a Lie algebra morphism, then it is not necessarily equivariant. This is true, however, if the acting group is compact and connected. Nonetheless, we will abuse terminology and call a moment map equivariant if it is a Lie group homomorphism.
\end{remark}
}

Now we turn our attention towards the multisymplectic setting. As in symplectic geometry, the infinitesimal equivariance of a weak moment map comes from differentiating $\psi_{\exp(t\xi),p}$ for fixed $\xi\in\g$ and $p\in\Rho_{\g,k}$.

\begin{proposition}\label{comp of Sigma} Let $\Sigma_k$ denote $\left.\frac{d}{dt}\right|_{t=0}\psi_{\exp(t\xi),p}$. Then we have that $\Sigma_k$ is a map from $\g$ to $\Rho_{\g,k}^\ast\otimes \Omega^{n-k}_{\mathrm{cl}}$ and is given by \[\Sigma_k(\xi,p)=f_k([\xi,p])+L_{V_\xi}f_k(p).\]

\end{proposition}

\begin{proof}We have that 
\begin{align*}
\left.\frac{d}{dt}\right|_{t=0}\psi_{\exp(t\xi),p}&=\left.\frac{d}{dt}\right|_{t=0} f_k(p)-\left.\frac{d}{dt}\right|_{t=0}\Phi_{\exp(-t\xi)}^\ast(f_k(\mathrm{Ad}_{\exp(-t\xi)}(p)))\\
&=-\left.\frac{d}{dt}\right|_{t=0}\Phi^\ast_{\exp(-t\xi)}(f_k(\mathrm{Ad}_{\exp(-t\xi)}p))\\
&=-f_k(\left.\frac{d}{dt}\right|_{t=0}\mathrm{Ad}_{\exp(-t\xi)}p)-(\left.\frac{d}{dt}\right|_{t=0}\Phi^\ast_{\exp(-t\xi)})(f_k(p))\\
&=-f_k(-[\xi,p])+L_{\xi_M}f_k(p)&\text{by Corollary \ref{ad}.}\\
\end{align*}

\end{proof}

\del{ 
The next corollary is a complete generalization of the situation in symplectic geometry. 

\begin{corollary}
If a refined homotopy moment map is equivariant, then it is a morphism between the differential graded Lie algebras $(\Lambda^\bullet\g, \partial,[\cdot,\cdot])$ and $(\Omega^\bullet_p(M),d,\{\cdot,\cdot\})$. (Change this is a little, this is true but it is not a corollary of the above)
\end{corollary}

\begin{proof}
If $(f)$ is equivariant, then $\sigma=0$ and so $\Sigma=0$. TO DO
\end{proof}

}

Let $R_k=\Rho_{\g,k}^\ast\otimes \Omega^{n-k}_\text{cl}$. Then $R_k$ is a $\g$-module under the induced action from the tensor product of the adjoint and Lie derivative actions. Concretely, for $\alpha\in R_k$, $\xi\in\g$ and $p\in\Rho_{\g,k}$, \[(\xi\cdot\alpha)(p)=\alpha([\xi,p])+L_{V_\xi}\alpha.\]
Consider the cohomology complex \[R_k\to C^1(\g,R_k)\to C^2(\g,R_k)\to\cdots,\]where the differential is the usual one from equation (\ref{group differential 2}).

The following is a generalization of Proposition \ref{inf cocycle}.

\begin{proposition}\label{general inf}
The map $\Sigma_k$ is in the kernel of $\partial_k$. That is, $\Sigma_k$ is a cocycle. \end{proposition}

\begin{proof}We need to show that $\partial\Sigma_k=0$. Indeed, for $\xi,\eta\in\g$ and $p\in\Rho_{\g,k}$, we have that 
\begin{align*}
\partial\Sigma_k(\xi,\eta)(p)&=\xi\cdot(\Sigma_k(\eta)(p))-\eta\cdot(\Sigma_k(\xi)(p))+\Sigma_k([\xi,\eta])(p)&\text{by equation (\ref{group differential 2})}\\
&= \Sigma_k(\eta)(\mathrm{ad}_\xi(p))+\L_{V_\xi}(\Sigma_k(\eta)(p))-\Sigma_k(\xi)(\mathrm{ad}_\eta(p))\\
&\quad{}-\L_{V_\eta}(\Sigma_k(\xi)(p))+\Sigma_k([\xi,\eta])(p)\\
&=\Sigma_k(\eta)([\xi,p])+\L_{V_\xi}(\Sigma_k(\eta)(p))-\Sigma_k(\xi)([\eta,p])\\
&\quad{}-\L_{V_\eta}(\Sigma_k(\xi)(p))+\Sigma_k([\xi,\eta])(p)&\text{by definition of ad}\\
&=f_k([\eta,[\xi,p]])+\L_{V_\eta}f_k([\xi,p])+\L_{V_\xi}f_k([\eta,p])+\L_{V_\xi}\L_{V_\eta}f_k(p)\\
&\quad{}-f_k([\xi,[\eta,p]])-\L_{V_\xi}f_k([\eta,p])-\L_{V_\eta}f_k([\xi,p])-\L_{V_\eta}\L_{V_\xi}f_k(p)\\
&\quad{}+f_k([[\xi,\eta],p])-\L_{V_{[\xi,\eta]}}f_k(p)\\
&=f_k([\eta,[\xi,p]])-f_k([\xi,[\eta,p]])+f_k([[\xi,\eta],p])\\
&\quad{}+\L_{V_\xi}\L_{V_\eta}f_k(p)-\L_{V_\eta}\L_{V_\xi}f_k(p)-\L_{V_{[\xi,\eta]}}f_k(p)\\
&=\L_{V_\xi}\L_{V_\eta}f_k(p)-\L_{V_\eta}\L_{V_\xi}f_k(p)-\L_{V_{[\xi,\eta]}}f_k(p)&\text{by the Jacobi identity}\\
&=0&\text{by the Lie derivative property}.
\end{align*}

\del{For $\xi,\eta\in\g$ and $p\in\Rho_{\g,k}$ we have by definition that $\partial\Sigma_k(\xi,\eta)(p)=\Sigma_k(\eta, ad_\xi(p))-\Sigma_k(\xi,ad_\eta(p))+\Sigma_k(\xi,\eta)(p)$. That is, showing that $\partial\Sigma_k=0$ is equivalent to showing that $\Sigma$ satisfies the Jacobi identity \[\Sigma_k([\eta,[\xi,p]])+\Sigma_k([\xi,[p,\eta]])+\Sigma_k([p,[\xi,\eta]])=0.\]

Indeed we have that

\begin{align*}
\Sigma_k([\eta,[\xi,p]])+&\Sigma_k([\xi,[p,\eta]])+\Sigma_k([p,[\xi,\eta]])=f_k([\eta,[\xi,p]]+[\xi,[p,\eta]]+[p,[\xi,\eta]])\\
&+\{f_1(\eta),f_k[\xi,p]\}+\{f_1(\xi),f_k[\eta,p]\}+\{f_1([\xi,\eta]),f_k(p)\}\\
&+d(\eta\hk f_k([\xi,p]))+d(\xi\hk f_k([\eta,p]))+d([\xi,\eta]\hk f_k(p))\\
&=\{f_1(\eta),f_k[\xi,p]\}+\{f_1(\xi),f_k[\eta,p]\}+\{f_1([\xi,\eta]),f_k(p)\}&\text{by Jacobi}\\
&+d(\eta\hk f_k([\xi,p]))+d(\xi\hk f_k([\eta,p]))+d([\xi,\eta]\hk f_k(p))\\
&=\{f_1(\eta),\{f_k(p),f_1\xi\}]\}+\{f_1(\xi),\{f_k(p),f_1(\xi)\}\}+\{\{f_1(\xi),f_1(\eta)\},f_k(p)\}\\
&+d(\eta\hk f_k([\xi,p]))+d(\xi\hk f_k([\eta,p]))+d([\xi,\eta]\hk f_k(p))\\
&=\L_\eta(f_k([\xi,p]))+\L_\xi(f_k([\eta,p]))+\L_{[\xi,\eta]}f_k(p)&\text{by prop \ref{Jacobi}}\\
&=0&\text{since $im(f)\subset\Omega^\bullet_{Sc}(M)$}
\end{align*}}

\end{proof}

As in symplectic geometry, we have that for a connected Lie group, a weak homotopy moment map is equivariant if and only if it is infinitesimally equivariant. That is, the weak homotopy $k$-moment map is equivariant if and only if $\sigma_k=0$ or $\Sigma_k=0$. A weak homotopy moment map is equivariant if $\sigma_k=0$ or $\Sigma_k=0$ for all $1\leq k\leq n$.

Now that we have generalized the notions of equivariance from symplectic to multisymplectic geometry, we move on to study the existence and uniqueness of these weak homotopy moment maps.
 
\section{Existence of Not Necessarily Equivariant Weak Moment Maps}

In this section we show how the results on the existence of not necessarily equivariant moment maps in symplectic geometry generalizes to multisymplectic geometry.

For a connected Lie group $G$ acting symplectically on a symplectic manifold $(M,\omega)$, recall the following standard results from symplectic geometry. 

\begin{proposition}\label{bracket gives}
For any $\xi,\eta\in\g$ we have \[[V_\xi,V_\eta]\hk\omega=d(V_\xi\hk V_\eta\hk\omega).\]
\end{proposition}

\del{
\begin{proposition}\label{H1}
We have that $H^1(\g)=[\g,\g]^0$, where $[\g,\g]^0$ is the annihilator of $[\g,\g]$.
\end{proposition}

\begin{proof}
This follows since for $c\in\g^\ast$ we have $\partial c(\xi,\eta)=c([\xi,\eta])$ by definition.
\end{proof}
}

\begin{proposition}\label{H2}
We have that $H^1(\g)=0$ if and only if $\g=[\g,\g]$.
\end{proposition}

and 

Combining these two propositions

\begin{proposition}\label{H4}
If $H^1(\g)=0$, then any symplectic action admits a moment map, which is not necessarily equivariant.
\end{proposition}

We now show how these results generalize to multisymplectic geometry. Let a connected Lie group act multisymplectically on an $n$-plectic manifold $(M,\omega)$.
\begin{proposition}\label{existence 1} 
For arbitrary $q$ in $\Rho_{\g,k}$ and $\xi\in\g$ we have that \[[V_q,V_\xi]\hk\omega= -(-1)^kd(V_q\hk V_\xi\hk\omega).\] 
\end{proposition}

\begin{proof}
By linearity it suffices to consider decomposable $q=\eta_1\wedge\cdots\wedge\eta_k$. A quick computation shows that $[V_q,V_\xi]\hk\omega=-V_{[q,\xi]}\hk\omega$. It follows that

\begin{align*}
V_{[q,\xi]}\hk\omega&=V_{\partial({q\wedge\xi})}\hk\omega&\text{by Proposition \ref{wedge is Schouten}}\\
&=(-1)^kd(V_{q\wedge\xi}\hk\omega)-\sum_{i=1}^k(-1)^i\eta_1\wedge\cdots\wedge\widehat \eta_i\wedge\cdots\wedge 
\eta_k\wedge\xi\hk\L_{\eta_i}\omega -V_{q\wedge\xi}\hk d\omega&\text{by Lemma \ref{extended Cartan}}\\
&=(-1)^kd(V_q\hk V_\xi\hk\omega).\\
\end{align*}The claim now follows.
\end{proof}
The next proposition is a generalization of Proposition \ref{H2}.

\begin{proposition}\label{kernel equals bracket}
If $H^0(\g,\Rho_{\g,k}^\ast)=0$ then $\Rho_{\g,k}=[\Rho_{\g,k},\g]$.
\end{proposition}
\begin{proof}
By equation (\ref{group differential 2}), an element $c\in H^0(\g,\Rho_{\g,k}^\ast)$ satisfies $c([\xi,p])=0$ for all $\xi\in\g$.  That is, \[H^0(\g,\Rho_{\g,k}^\ast)=[\Rho_{\g,k},\g]^0,\] where $[\Rho_{\g,k},\g]^0$ is the annihilator of $[\Rho_{\g,k},\g]$.
\end{proof}
We now arrive at our main theorem on the existence of not necessarily equivariant weak moment maps. The following is a generalization of Proposition \ref{H4}.

\begin{theorem}\label{theorem existence}
Let $G$ act multisymplectically on $(M,\omega)$. If $H^0(\g,\Rho_{\g,k}^\ast\otimes\Omega^{n-k}_{\mathrm{cl}})=0$, and $H^0(\g,\Omega^{n-k}_{\mathrm{cl}})\not=0$, then the $k$-th component of a not necessarily equivariant moment map exists.
\end{theorem}

\begin{proof}

If $H^0(\g,\Rho_{\g,k}^\ast\otimes\Omega^{n-k}_{\mathrm{cl}})=0$ and $H^0(\g,\Omega^{n-k}_{\mathrm{cl}})\not=0,$  then $H^0(\g,\Rho_{\g,k}^\ast)=0$ by the Kunneth formula (see for example Theorem 3.6.3 of \cite{kunneth}). The claim now follows from Proposition \ref{kernel equals bracket} and Proposition \ref{existence 1}. Indeed, Proposition \ref{existence 1} says we may define a weak moment map on elements of the form $[p,\xi]$ by $(-1)^kV_p\hk V_\xi\hk\omega$, where $p\in\Rho_{\g,k}$ and $\xi\in\g$, and Proposition \ref{kernel equals bracket} says every element in $\Rho_{\g,k}$ is a sum of elements of this form.
\end{proof}

\begin{remark}
Notice that for the case $n=k$, it is always true that $H^0(\g,\Omega^{n-k}_{\mathrm{cl}})\not=0$ since any-non zero constant function is closed. Hence Theorem \ref{theorem existence} gives a generalization of Theorems 3.5 and 3.14 of \cite{ms} and \cite{MS} respectively. Moreover, by taking $n=k=1$, we see that we are obtaining a generalization from symplectic geometry.
\end{remark}

\begin{example}\label{example 1}
Consider the multisympletic manifold $(\R^4,\omega)$ where $\omega=\mathrm{vol}$ is the standard volume form. That is, we are working in the case $n=3$. Let $x_1,\cdots,x_4$ denote the standard coordinates. Let $G=U(2)$ act on $\R^4$ by rotations. The corresponding Lie algebra action generates the vector fields \[E_0=x^3\frac{\pd}{\pd x^1}+x^4\frac{\pd}{\pd x^2}-x^1\frac{\pd}{\pd x^3}-x^2\frac{\pd}{\pd x^4},\] \[E_1=x^3\frac{\pd}{\pd x^1}+x^4\frac{\pd}{\pd x^2}-x^1\frac{\pd}{\pd x^3}-x^2\frac{\pd}{\pd x^4},\] \[E_2=-x^2\frac{\pd}{\pd x^1}+x^1\frac{\pd}{\pd x^2}-x^4\frac{\pd}{\pd x^3}+x^3\frac{\pd}{\pd x^4},\] and \[E_3=x^4\frac{\pd}{\pd x^1}+x^3\frac{\pd}{\pd x^2}-x^2\frac{\pd}{\pd x^3}-x^1\frac{\pd}{\pd x^4}.\] For the case $k=2$, consider the distance function $r=\sqrt{x_1^2+x_2^2+x_3^2+x_4^2}$. It is clear that the distance function is invariant under rotations and hence $\L_{E_i}dr=0$ for $i=0,1,2,3$. Since $dr$ is a closed $1$-form, it follows that $dr$ is a non-zero element of $H^0(\g,\Omega^{1}_{\mathrm{cl}}(M))$. That is, $H^0(\g,\Omega^{1}_{\mathrm{cl}}(M))\not=0$.

For the case $k=1$, consider $\alpha:=dx^1\wedge dx^2+dx^3\wedge dx^4$. A quick calculuation shows that $E_i\hk\alpha=$ for $i=0,1,2,3$ so that $\alpha$ is invariant under the $\mathfrak{u}(2)$ action. Since $d\alpha=0$, it follows that $H^0(\g,\Omega^{2}_{\mathrm{cl}}(M))\not=0$ as well.

Hence, by Theorem \ref{theorem existence}, it follows that a weak moment map exists. 
\end{example}
The next example gives a scenario for which Theorem \ref{theorem existence} can only be applied to specific components of a weak moment map. 
\begin{example}Take the setup of Example \ref{example 1} but instead consider the action of $SO(4)$.  As in Example \ref{example 1}, $dr$ is a non-zero closed $1$-form which is invariant under the action. That is, $H^0(\g,\Omega^{1}_{\mathrm{cl}}(M))\not=0$. However, in this setup, $H^0(\g,\Omega^{2}_{\mathrm{cl}}(M))=0$. Indeed, the infinitesimal generators of $\mathfrak{so}(4)$ are of the form $x^i\frac{\pd}{\pd x^j}-x^j\frac{\pd}{\pd x^i}$ where $1\leq i,j,\leq 4$. An arbitrary $2$-form may be written as $\beta=\sum_{i,j}a_{ij}dx^i\wedge dx^j$. A computation shows that the condition $\L_{V_\xi}\beta=0$ for all $\xi\in\mathfrak{so}(4)$ showing that necessarily $\beta=0$. Hence  $H^0(\g,\Omega^{1}_{\mathrm{cl}}(M))=0$. 

It follows that, in this case, Theorem \ref{theorem existence} guarantees the existence of the $2$nd component of a weak moment map, but does not guarantee the existence of the $1$st.

\end{example}
Another generalization of Proposition \ref{H4} to multisymplectic geometry is given by:

\begin{proposition}\label{dont know}
If $H^k(\g)=0$, then the $k$-th component of a not necessarily equivariant weak moment map exists.
\end{proposition}

\begin{proof}
If $H^k(\g)=0$ then $\Rho_{\g,k}=\mathrm{Im}(\partial_{k+1})$, since $\Rho_{\g,k}=\mathrm{ker}(\partial_k)$. But for $p\in\mathrm{Im}(\partial_{k+1})$ we have that $p=\partial q$ for some $q\in\Lambda^{k+1}\g$. Then by Lemma \ref{extended Cartan} we have 

\[V_p\hk\omega=(-1)^kd(V_q\hk\omega).\] Hence we may define $f_k(p)$ to be $(-1)^kV_q\hk\omega$.
\end{proof}

\begin{remark}
Proposition \ref{dont know} gives another generalization of the results of Madsen and Swann. Indeed, by taking $n=k$ we again arrive at Theorems 3.5 and 3.14 of \cite{ms} and \cite{MS} respectively.
\end{remark}

Summarizing Theorem \ref{theorem existence} and Proposition \ref{dont know} we obtain:

\begin{proposition}\label{theorem open question 1}
If $H^1(\g)=\cdots=H^n(\g)=0$ then a not necessarily equivariant weak moment map $(f)$ exists.
\end{proposition}

\begin{theorem}
If, for all $1\leq k\leq n$,  $H^0(\g,\Rho_{\g,k}^\ast)=0$, or equivalently $H^0(\g,\Rho_{\g,k}^\ast\otimes\Omega^{n-k}_{\mathrm{cl}})=0$ and $H^0(\g,\Omega^{n-k}_{\mathrm{cl}})\not=0$, then a not necessarily equivariant weak moment map $(f)$ exists.
\end{theorem}

\del{

\begin{proposition}\label{existence 2} If $H_k(\g)=0$ then every element $p$ in $\Rho_{\g,k}$ is of the form \[p=\sum[q_i,\xi_i],\]where $q_i\in\Rho_{\g,k}$ and $\xi_i\in\g$. 
\end{proposition}

\begin{proof}
This fact follows from Theorem 3.5 of Madsen and Swann, although have question about their argument.

 Don't know if this helps but we have that $[\Rho_{\g,k},\g]$ is a subset of $\Rho_{\g,k}$. Let $[p,\xi]$ be an arbitrary element of $[\Rho_{\g,k},\g]$.
\begin{align*}
\partial(p\wedge\xi)&=\partial(p)\wedge\xi+p\wedge\partial(\xi)+[p,\xi]\\
&=p\wedge\partial(\xi)+[p,\xi]&\text{since $\partial(p)=0$}\\
&=[p,\xi]&\text{since $\xi$ has degree $1$}
\end{align*}
Thus, $[p,\xi]$ is closed.
\end{proof} 
Combining these two propositions gives us the main theorem of this section.
\begin{proposition}
Let $\g$ act multisymplectically on a manifold $M$. If $H_k(\g)=0$, then the $k$-th component, $f_k$, of a moment map exists. The map $f_k$ is not necessarily equivariant.
\end{proposition}

\begin{proof}
This follows directly from Propositions \ref{existence 1} and \ref{existence 2}.
\end{proof}

\begin{proposition}
Let $\g$ act multisymplectically on $(M,\omega)$. If $H^1(\g)=\cdots=H^n(\g)=$ then a not necessarily equivariant weak homotopy moment map exists.
\end{proposition}
}
In the next section we study when a non-equivariant weak moment map can be made equivariant.

\section{Obtaining an Equivariant Moment Map from a Non-Equivariant Moment Map}

In this section we show that the theory involved in obtaining an equivariant moment map from a non-equivariant moment map extends from symplectic to multisymplectic geometry. We first recall the standard results from symplectic geometry.

Proposition \ref{inf cocycle} shows that the map $\Sigma$ corresponding to a moment map $f$ is a Lie algebra $2$-cocycle. The next proposition says that if the cocycle is exact then $f$ can be made equivariant.

\begin{proposition}\label{if exact}
Let $f$ be a moment map and $\Sigma$ its corresponding cocycle. If $\Sigma=\partial(l)$ for some $l$, then $f+l$ is equivariant,
\end{proposition}

It follows from this that 
\begin{proposition}\label{obtain theorem}
If $H^2(\g)=0$ then one can obtain an equivariant moment map from a non-equivariant moment map.
\end{proposition}

Now let $G$ be a connected Lie group acting on an $n$-plectic manifold $(M,\omega)$. The following proposition generalizes Proposition \ref{if exact} to multisymplectic geometry.
\begin{proposition}\label{exact Sigma}
Let $f_k$ be the weak homotopy $k$-moment map, and let $\Sigma_k$ denote its corresponding cocycle. If $\Sigma_k=\partial(l_k)$ for some $l_k\in H^0(\g,\Rho_{\g,k}^\ast\otimes\Omega^{n-k}_{\mathrm{cl}})$, then $f_k+l_k$ is equivariant.
\end{proposition}

\begin{proof}
Fix $p\in\Rho_{\g,k}$ and $\xi\in\g$. Then

\begin{align*}
(f_k+l_k)([\xi,p])&=f_k([\xi,p])+l_k([\xi,p])\\
&=f_k([\xi,p])-((\partial l_k)(\xi))(p)+\L_{V_\xi}l_k(p)&\text{by equation (\ref{group differential 2})}\\
&=f_k([\xi,p])-\Sigma_k([\xi,p])+\L_{V_\xi}l_k(p)\\
&=\L_{V_\xi}f_k(p)+\L_{V_\xi}(l_k(p))&\text{by definition of $\Sigma_k$}\\
&=\L_{V_\xi}((f_k+l_k)(p)).
\end{align*}

\end{proof}
We now arrive at our generalization of Proposition \ref{obtain theorem}:

\begin{theorem}
If $H^1(\g,\Rho_{\g,k}^\ast\otimes\Omega^{n-k}_{\mathrm{cl}})=0$ then any weak $k$-moment map can be made equivariant. In particular, if $H^1(\g,\Rho_{\g,k}^\ast\otimes\Omega^{n-k}_{\mathrm{cl}})=0$ for all $1\leq k \leq n$, then any weak moment map $(f)$ can be made equivariant.
\end{theorem}

\begin{proof}
Let $f_k:\Rho_{\g,k}\to\Omega^{n-k}_\mathrm{Ham}$ be a weak $k$-moment map. If $H^1(\g,\Rho_{\g,k}^\ast\otimes\Omega^{n-k}_{\mathrm{cl}})=0$ then the corresponding cocycle $\Sigma_k$ is exact, i.e. $\Sigma_k=\partial(l_k)$ for some $l_k\in H^0(\g,\Rho_{\g,k})$. It follows from Proposition \ref{exact Sigma} that $f_k+l_k$ is equivariant.
\end{proof}

\section{Uniqueness of Weak moment Maps}

We first recall the results from symplectic geometry without explicit proof. A proof can be found by setting $n=1$ (i.e. the symplectic case) in our more general Theorem \ref{dunno2}.
Let $\g$ be a Lie algebra acting on a symplectic manifold $(M,\omega)$.
\begin{proposition}\label{H5}
If $f$ and $g$ are two equivariant moment maps, then $f-g$ is in $H^1(\g)$.
\end{proposition}
\begin{proof}
For $\xi,\eta\in\g$ we have that $(f-g)([\xi,\eta])=\{(f-g)(\xi),(f-g)(\eta)\}$ since $f$ and $g$ are equivariant. However, $(f-g)(\xi)$ is a constant function since both $f$ and $g$ are moment maps. The claim now follows since the Poisson bracket with a constant function vanishes.
\end{proof}
From Proposition \ref{H5} it immediately follows that
\begin{proposition}\label{H6}
If $H^1(\g)=0$ then equivariant moment moments are unique.
\end{proposition}

The following is a generalization of Proposition \ref{H5}.
\begin{proposition}\label{uniqueness in H^0}
If $f_k$ and $g_k$ are $k$-th components of two equivariant weak moment maps, then $f_k-g_k$ is in $H^0(\g,\Rho_{\g,k}^\ast\otimes\Omega^{n-k}_{\mathrm{cl}})$.
\end{proposition}

\begin{proof}
If $f_k$ and $g_k$ are equivariant then $(f_k-g_k)([\xi,p])=\L_{V_\xi}((f_k-g_k)(p))$. Moreover, $(f_k-g_k)(p)$ is closed since both $f_k$ and $g_k$ are moment maps.
\end{proof}

We now arrive at our generalization of Proposition \ref{H6}. Let $\g$ be a Lie algebra acting on an $n$-plectic manifold $(M,\omega)$.
\begin{theorem}\label{dunno2}
If $H^0(\g,\Rho_{\g,k}^\ast\otimes\Omega^{n-k}_{\mathrm{cl}})=0$, then equivariant weak $k$-moment maps are unique. In particular, if $H^0(\g,\Rho_{\g,k}^\ast\otimes\Omega^{n-k}_{\mathrm{cl}})=0$ for all $1\leq k \leq n$ then equivariant weak moment maps are unique. 
\end{theorem}
\begin{proof}
If $f_k$ and $g_k$ are two equivariant weak $k$-moment maps, then Proposition \ref{uniqueness in H^0} shows that $f_k-g_k$ is in $H^0(\g,\Rho_{\g,k}^\ast\otimes\Omega^{n-k}_{\mathrm{cl}})$.

\end{proof}

\begin{remark}
This theorem gives a generalization of the results of Madsen and Swann. Indeed, by taking $n=k$ we again arrive at Theorems 3.5 and 3.14 of \cite{ms} and \cite{MS} respectively.
\end{remark}
\del{
We make two more observations.

\begin{proposition}\label{HH1}
If $f_k$ and $g_k$ are two components of an equivariant homotopy moment map $(f)$, then $f_k-g_k$ is in $[\Rho_{\g,k},\g]^0$. 
\end{proposition}
\begin{proof}
Let $[p,\xi]$ be an artbitrary element of $[\Rho_{\g,k},\g]$. Then we have that 
\begin{align*}
(f_k-g_k)([p,\xi])&=\{(f_k-g_k)(p),(f_1-g_1)(\xi)\}&\text{by equivariance}\\
&=0&\text{since $f_k(p)-g_k(p)$ is closed}
\end{align*}
\end{proof}
\begin{proposition}\label{HH2}
If $H^{k}(\g)=0$ then $[\Rho_{\g,k},\g]^0=0$
\end{proposition}

\begin{proof}
We have already shown that $\Rho_{\g,k}=[\Rho_{\g,k},\g]$. Since $\dim([\Rho_{\g,k},\g])+\dim([\Rho_{\g,k},\g])^0=\dim(\Rho_{\g,k})$ the claim now follows.
\end{proof}

Our generalization of Proposition \ref{H5} is:

\begin{proposition}
If $H^k(g)=0$ then the $k$-th components of homotopy moment maps are unique. 
\end{proposition}

\begin{proof}
This follows from Propositions \ref{HH1} and \ref{HH2}.
\end{proof}

Therefore, if an equivariant moment map exists and $H^k(\g)=0$, it is unique.
}

\section{Weak moment Maps as Morphisms}

Consider a symplectic action of a connected Lie group $G$ acting on a symplectic manifold $(M,\omega)$. Let $f:\g\to C^\infty(M)$ be a moment map. By Definition \ref{inf equiv moment}, $f$ is equivariant if and only if $f$ is a Lie algebra morphism from $(\g,[\cdot,\cdot])$ to $(C^\infty(M),\{\cdot,\cdot\})$. That is, if and only if \[f([\xi,\eta])=\{f(\xi),f(\eta)\}.\] Taking $d$ of both sides of this equation yields:
\begin{proposition}\label{morph 1}
A moment map $f$ induces a morphism onto the quotient of $C^\infty(M)$ by constant functions. That is, a moment map induces a Lie algebra morphism from $(\g,[\cdot,\cdot])$ to $(C^\infty(M)/\text{constant},\{\cdot,\cdot\})$, regardless of equivariance. If moreover, the moment map $f$ is equivariant, then $f$ is a morphism from $(\g,[\cdot,\cdot])$ to $(C^\infty(M),\{\cdot,\cdot\})$.

\end{proposition}

We now restate Proposition \ref{morph 1} in an equivalent way, but which will allow for a direct generalization to multisymplectic geometry: Notice that $\g$ is a $\g$-module under the Lie bracket action and $C^\infty(M)$ is $\g$-module under the action $\xi\cdot\g=L_{V_{\xi}}g$, where $\xi\in\g$ and $g\in C^\infty(M)$. Proposition \ref{morph 1} is equivalent to:

\begin{proposition}\label{Morph 1}
A moment map  $f$ always induces a $\g$-module morphism from $\g$ to $C^\infty(M)/\text{constant}$. Moreover, if the moment map $f$ is equivariant, then it is a $\g$-module morphism from $\g$ to $C^\infty(M)$.
\end{proposition}

Now let a connected Lie group $G$ act multisymplectically on an $n$-plectic manifold $(M,\omega)$. 

\begin{proposition}
For any $1\leq k\leq n$, we have that $\Rho_{\g,k}$ is a $\g$-module under the action $\xi\cdot p = [p,\xi]$, where $p\in\Rho_{\g,k}$, $\xi\in\g$, and $[\cdot,\cdot]$ is the Schouten bracket. 
\end{proposition}

\begin{proof}
This follows since Proposition \ref{wedge is Schouten} shows that $[p,\xi]$ is in the Lie kernel.
\end{proof}

\begin{proposition}
For any $1\leq k \leq n$, we have that $\Omega^{n-k}_{\mathrm{Ham}}(M)$ is a $\g$-module under the action $\xi\cdot\alpha=\L_{V_\xi}\alpha$, where $\alpha\in\Omega^{n-k}_{\mathrm{Ham}}(M)$ and $\xi\in\g$.
\end{proposition}

\begin{proof}
Suppose that $\alpha\in\Omega^{n-k}_{\mathrm{Ham}}(M)$ is a Hamiltonian $(n-k)$-form. Then $d\alpha=-X_\alpha\hk\omega$ for some $X_\alpha\in\Gamma(\Lambda^k(TM))$. Then, for $\xi\in\g$,

\begin{align*}
d\L_{V_\xi}\alpha&=-\L_{V_\xi}(X_\alpha\hk\omega)\\
&=-\L_{V_\xi}(X_\alpha\hk\omega)+X_\alpha\hk\L_{V_\xi}\omega &\text{since $\L_{V_\xi}\omega=0$}\\
&=[V_\xi,X_\alpha]\hk\omega &\text{by the product rule}
\end{align*}

Hence $\L_{V_\xi}\alpha$ is in $\Omega^{n-k}_{\mathrm{Ham}}(M)$.

\end{proof}

Our generalization of Proposition \ref{Morph 1} to multisymplectic geometry is:

\begin{theorem}\label{morphism 1}
For any $1\leq k \leq n$, the $k$-th component of a moment map $f_k$ is a $\g$-module morphism from $\Rho_{\g,k}$ to $\Omega^{n-k}_{\mathrm{Ham}}(M)/\mathrm{closed}$. Moreover, a weak $k$-moment map $f_k$ is equivariant if and only if it is a $\g$-module morphism from $\Rho_{\g,k}$ to $\Omega^{n-k}_{\mathrm{Ham}}(M)$.
\end{theorem}

\begin{proof}
Suppose that $(f)$ is a weak moment map. Then, by definition
\begin{align*}
df_k([\xi,p])&=-\zeta(k)V_{[\xi,p]}\hk\omega\\
&=-\zeta(k)[V_\xi,V_p]\hk\omega\\
&=-\zeta(k)\L_{V_\xi}(V_p\hk\omega)\\
&=\zeta(k)\zeta(k)d\L_{V_\xi}f_k(p)\\
&=d\L_{V_\xi}f_k(p).
\end{align*}
 This proves the first statement of the theorem. Now suppose $f_k$ is equivariant. It follows that $\Sigma_k=0$. Thus, by Proposition \ref{comp of Sigma} we have $f_k([\xi,p])=\L_{V_\xi}f_k(p)$. Conversely, if $f_k$ is a $\g$-module morphism, that $f_k([\xi,p])=\L_{V_\xi}f_k(p)$ for every $\xi\in\g$ and $p\in\Rho_{\g,k}$. That is, $\Sigma_k=0$. 

\end{proof}
\section{Open Questions}
\del{
\subsection{Symplectic Geometry}
\begin{proposition}
If $H^1(\g)=0=H^2(\g)$ then a unique equivariant moment map exists.
\end{proposition}

\subsection{Multisymplectic Geometry}

\begin{theorem}
If $H^0(g,\Rho_{\g,k}^\ast\otimes\Omega^{n-k}_{\text{cl}})=0$ and $H^1(g,\Rho_{\g,k}^\ast\otimes\Omega^{n-k}_{\text{cl}})=0$, then there exists a unique equivariant weak homotopy $k$-moment map.
\end{theorem}

\begin{theorem}
If $H^0(g,\Rho_{\g,k}^\ast\otimes\Omega^{n-k}_{\text{cl}})=0$ and $H^1(g,\Rho_{\g,k}^\ast\otimes\Omega^{n-k}_{\text{cl}})=0$ for all $1\leq k\leq n$, then a full equivariant weak moment map exists uniquely.
\end{theorem}

\begin{theorem}
If $H^k(\g)=0$ then a 
\end{theorem}

}
We end by noting some open questions naturally posed by the results in this paper.

\begin{enumerate}\del{
\item Fix $n\geq 1$. In light of Theorem \ref{theorem open question 1}, it is natural to ask whether there exists a Lie algebra $\g$ such that $H^k(\g)=0$ for all $1\leq k \leq n$?}
\item Consider Theorems \ref{morphism 1}. In symplectic geometry, Proposition \ref{inf cocycle} shows that a moment map $f:\g\to C^\infty(M)$ induces a Lie algebra morphism from $(\g,[\cdot,\cdot])$ to the quotient space $(C^\infty(M)/\mathrm{constant},\{\cdot,\cdot\})$, and if $f$ is equivariant then it is a Lie algebra morphism from $(\g,[\cdot,\cdot])$ to $(C^\infty(M)/\mathrm{exact},\{\cdot,\cdot\})$. Moreover, in \cite{me},
Proposition 4.10 showed that both $\Omega^\bullet_{\mathrm{Ham}}(M)/\mathrm{closed}$ and $\Omega^\bullet_{\mathrm{Ham}}(M)/\mathrm{exact}$ are graded Lie algebras while Proposition 5.9 of \cite{me} showed that a weak homotopy moment map is always a graded Lie algebra morphism from $\Rho_\g$ to $\Omega^\bullet_{\mathrm{Ham}}(M)/closed$. 

\hspace{0.2cm} Hence, a natural question is:

If $(f)$ is an equivariant weak moment map, does it induce a graded Lie algebra morphism from $(\Rho_\g,[\cdot,\cdot])$ to $(\Omega^\bullet_{\mathrm{Ham}}(M)/\mathrm{exact},\{\cdot,\cdot\})$? Conversely?

\item In  our work, we provided a couple of examples of $n$-plectic group actions to which our theory of the existence and uniqueness of moment maps could be applied. There are many other interesting $n$-plectic geometries; see for example \cite{questions}, \cite{me} and \cite{cq}. What does the work done in our paper say about the existence and uniqueness of moment maps in these setups?

\item Given a weak moment map $(f)$ with $f_k:\Rho_{\g,k}\to\Omega^{n-k}_{\mathrm{Ham}}(M)$, does there exists a full homotopy moment map $(h)$ whose restriction to the Lie kernel is $(f)$? Something about equivariant cohomology. In particular, what is the relationship between the results on the existence and uniqueness of homotopy moment maps given in \cite{existence 1} and \cite{existence 2} to the results in this paper?

\end{enumerate}

\end{document}